\documentclass[11pt, reqno]{amsart}
\usepackage{fullpage}
\usepackage{amsmath}
\usepackage{amssymb}
\usepackage{graphicx}
\usepackage{float}
\usepackage{algorithm,algorithmicx}
\usepackage{algpseudocode}
\usepackage{lscape}
\usepackage{amsthm}
\usepackage{booktabs}
\usepackage{multirow}
\usepackage{rotating}
\usepackage{geometry}
\usepackage{mathtools}
\usepackage{hyperref}
\usepackage{multirow}
\usepackage{array}
\usepackage[numbers]{natbib}
\usepackage{bcol}

\algrenewcomment[1]{\(\triangleright\) #1}

\usepackage[colorinlistoftodos,prependcaption,textsize=small]{todonotes}

\newtheorem{proposition}{Proposition}

\newtheorem{lemma}{Lemma}

\newtheorem{example}{Example}

\newcommand{\R}{\ensuremath{\mathbb{R}}}

\newcommand{\Z}{\ensuremath{\mathbb{Z}}}

\newcommand{\prob}{\operatorname{\mathbf{Prob}}}

\newcommand{\diag}{\operatorname{diag}}

\newcommand{\cR}{\mathcal{R}}

\newcommand{\mr}{\ensuremath{\text{MR}}}
\newcommand{\po}{\ensuremath{\text{PO}}}
\newcommand{\dmr}{\ensuremath{\text{DMR}}}
\newcommand{\pni}{\ensuremath{\text{MRNI}}}

\newcommand{\exclude}[1]{}
\newcommand{\BD}[1]{\normalsize{\mathbf{#1}}}

\def\1{\mathbf{1}}
\def\0{\mathbf{0}}

\DeclareMathOperator*{\argmin}{arg\,min}

\makeatletter
\newcommand{\leqnomode}{\tagsleft@true}
\newcommand{\reqnomode}{\tagsleft@false}
\makeatother

\usepackage[normalem]{ulem}
\newcommand{\ins}[1]{{#1}}
\newcommand{\rep}[2]{{#1}{}}
\newcommand{\del}[1]{}

\usepackage[colorinlistoftodos,prependcaption,textsize=small]{todonotes}

\author{
Alper Atamt\"urk, Carlos Deck and Hyemin Jeon 
}

\thanks{ \noindent \hskip -5mm
	A. Atamt\"urk: Department of Industrial Engineering \& Operations Research, University of California, Berkeley, CA 94720.
	\texttt{atamturk@berkeley.edu}   \\
	C. Deck:  Department of Industrial Engineering \& Operations Research, University of California, Berkeley, CA 94720.
	\texttt{cgdeck@berkeley.edu} \\ 
	H. Jeon: Department of Industrial Engineering \& Operations Research, University of California, Berkeley, CA 94720. 
	\texttt{hyemin.jeon@berkeley.edu }
}

\title[Quadratic Upper Bound for Discrete Mean-Risk Minimization]{Successive Quadratic Upper-Bounding for Discrete Mean-Risk Minimization and Network Interdiction}

\begin{document}

\begin{abstract}
	The advances in conic optimization have led to its increased utilization for modeling data uncertainty. In particular, conic mean-risk optimization gained prominence in probabilistic and robust optimization. Whereas the corresponding conic models are solved efficiently over convex sets, their discrete counterparts are intractable. In this paper, we give a highly effective successive quadratic upper-bounding procedure for discrete mean-risk minimization problems. The procedure is based on a reformulation of the mean-risk problem through the perspective of its convex quadratic term. Computational experiments conducted on the network interdiction problem with stochastic capacities show that the proposed approach yields solutions within 1-2\% of optimality in a small fraction of the time required by exact search algorithms. We demonstrate the value of the proposed approach for constructing efficient frontiers of flow-at-risk vs. interdiction cost for 
	varying confidence levels. \\

\end{abstract}

\maketitle

\begin{center}
	July 2017; \ins{June 2018}
\end{center}

\BCOLReport{17.05}

\pagebreak

\section{Introduction}

Conic optimization problems arise frequently when modeling 
parametric value-at-risk (VaR) minimization, portfolio optimization, and robust optimization with ellipsoidal objective uncertainty. Although convex versions of these models are solved efficiently by polynomial interior-point algorithms, their discrete counterparts are intractable. Branch-and-bound and branch-and-cut algorithms 
require excessive computation time even for relatively small instances. The computational difficulty is exacerbated by the lack of effective warm-start procedures for conic optimization. 

In this paper, we consider a reformulation of a conic quadratic discrete mean-risk minimization problem that lends itself to a successive quadratic optimization procedure benefiting from fast warm-starts and eliminating the need to solve conic optimization problems directly. 

Let $u$ be an $n$-dimensional random vector and $x$ be an $n$-dimensional decision vector in a closed set 
$X \subseteq \R^n$.
If $u$ is normally distributed with mean $c$ and covariance $Q$, 
the minimum value-at-risk for $u'x$ at confidence level $1-\epsilon \ $, i.e.,
		\begin{align*}
		\leqnomode
		\begin{split}
		\zeta(\epsilon) = \min \ & \bigg \{ z :  \prob \left( u'x > z \right) \leq \epsilon, \ \ x \in X \bigg \} , 
		\end{split}
		\end{align*}
for $0 < \epsilon \le 0.5$, is computed by solving the mean-risk optimization problem 
\begin{align*} 
(\mr) \ \ \ \min \ & \bigg \{  c' x + \Omega \displaystyle \sqrt{x'Q x} : \ x \in X \bigg \}, 
\end{align*}
where $\Omega = \Phi^{-1}(1-\epsilon)$ and $\Phi$ is the c.d.f. of the standard normal distribution \cite{Birge:SPbook}.
If $u$ is not normally distributed, but its mean and variance are known, (\mr) yields a 
robust version by letting $\Omega = \sqrt{(1-\epsilon)/\epsilon}$, which provides an upper bound on the worst-case VaR
\cite{bertsimas.popescu:05,ghaoui.etal:03}.
Alternatively, if $u_i$'s are independent and symmetric with support $[c_i - \sigma_i, c_i + \sigma_i]$, then letting $\Omega = \sqrt{\ln(1/\epsilon)}$ with $Q_{ii} = \sigma_i^2$ gives an upper bound on \del{on} the worst-case VaR as well \cite{BN:robust-mp}. The reader is referred to
\citet{RO-book} for an in-depth treatment of robust models through conic optimization.
Hence, under various assumptions on the uncertainty of $u$, one arrives at different instances of the mean-risk model (\mr) with a conic quadratic objective.
Ahmed \cite{ahmed:06} studies the complexity and tractability of various stochastic objectives for mean-risk optimization. Maximization of the mean-risk objective is \NP-hard even for a diagonal covariance matrix \citep{AA:utility,AG:max-util}. If $X$ is a polyhedron,
(\mr) is a special case of conic quadratic optimization \citep{Alizadeh2003,Lobo1998}, which can be solved by polynomial-time interior points algorithm \citep{Alizadeh1995,Nesterov1998,BTN:ModernOptBook}. 
\citet{AG:simplex-qp} give simplex QP-based algorithms for this case. 

The interest of the current paper is in the discrete case of (\mr)
with integrality restrictions: $X \subseteq \Z^n$, which is \NP-hard. 
\citet{AN:conicmir:ipco} describe mixed-integer rounding cuts, and 
\citet{CI:cmip} give disjunctive cuts for conic mixed-integer programming.
The integral case is more predominantly addressed in the special case of independent random variables over binaries. In the absence of correlations, the covariance matrix reduces to a diagonal matrix $Q = diag(q)$, where $q$ is the vector of variances. In addition, when the decision variables are binary, (\mr) reduces to
\begin{align*} 
(\dmr) \ \ \ \min \left\{c' x + \Omega \sqrt{q' x} :  x \in X \subseteq \mathbb{B}^n \right\} \cdot
\end{align*}

\ins{Several approaches are available for (\dmr) for specific constraint sets $X$.} 
\citet{Ishii1981} give an $O(n^6)$ algorithm \del{for (\dmr)} \rep{when the feasible set $X$ is the set of}{over} spanning trees; \citet{hassin1989maximizing} utilize parametric linear programming to solve \rep{(\dmr)}{it} \rep{when $X$ defines a matroid}{over matroids} in polynomial time. 
\citet{atamturk2008polymatroids} give a cutting plane algorithm utilizing the submodularity of the objective; \citet{AJ:lifted-polymatroid} extend it to the mixed 0-1 case with indicator variables.
\citet{Atamturk2009} give an $O(n^3)$ algorithm over a cardinality constraint.
\citet{shen2003joint} provide a greedy $O(n \log n)$ 
algorithm to solve the diagonal case over the unit hypercube.
 \citet{nikolova2009strategic} gives 
a \ins{fully polynomial-time approximation scheme} (FPTAS) for an arbitrary set $X \subseteq \mathbb{B}^n$ provided the deterministic problem with $\Omega=0$ can be solved in polynomial time. 

The reformulation we give in Section~\ref{sec:qp-ub} reduces the general discrete mean-risk problem (\mr) to a sequence of discrete quadratic optimization problems, which is often more tractable than the conic quadratic case \cite{AG:m-matrix}. The uncorrelated case (\dmr) reduces to a sequence of binary linear optimization problems. Therefore, one can utilize the simplex-based algorithms with fast warm-starts for general constraint sets. Moreover, the implementations can benefit significantly for structured constraint sets, such as spanning trees, matroids, graph cuts, shortest paths, for which efficient algorithms are known for non-negative linear objectives.

\subsubsection*{A motivating application: Network interdiction with stochastic capacities}
\label{sec:PNI}

Our motivating problem for the paper is network interdiction with stochastic capacities; although, since the proposed  approach is independent of the
feasible set $X$, it can be applied to any problem with a mean-risk objective (\mr).

The deterministic interdiction problem is a generalization of the classical min-cut problem, where an interdictor with a
limited budget minimizes the maximum flow on a network by stopping the flow on a subset of the arcs at a cost per interdicted arc. Consider a graph $G = (N,A)$ with nodes $N$ and arcs $A$. 
Let $s$ be the source node and $t$ be the sink node. Let $\alpha_a$ be the cost of interdicting arc $a \in A$ and $\beta$ be the total budget for interdiction. Then, given a set $I$ of interdicted arcs, the maximum $s-t$ 
flow on the remaining arcs is the capacity of the minimum cut on the arcs $A \setminus I$. \citet{wood1993deterministic} shows
that the deterministic interdiction problem is \NP-hard and gives integer formulations for it. 
\citet{royset2007solving} give algorithms for a bi-criteria interdiction problem and generate an 
efficient frontier of maximum flow vs. interdiction cost. 
\citet{cormican1998stochastic,janjarassuk2008reformulation} consider a stochastic version of the problem,
where interdiction success is probabilistic. \citet{HHW:stoc-interdiction} develop a decomposition approach
for interdiction when network topology is stochastic. Network interdiction
is a dual counterpart of survivable network design \cite{BM:surv-cuts,RA:review}, where one installs capacity to maximize the minimum flow against an adversary blocking the arcs. See \citet{Smith2013} for a review of network interdiction models and algorithms.

\ignore{
The decision variables for (DNI) based on graph $G$ are defined as follows: 
\begin{equation*} 
\begin{aligned}
x_{a} &= \left\{ 
\begin{array}{cl}
1 & \text{  if } a \in A \text{ is within cut} \\ 
0 & \text{  otherwise,}
\end{array} \right.\\
z_{a} &= \left\{ 
\begin{array}{cl}
1 & \text{  if } a \in A \text{ is within cut but removed by interdictor} \\ 
0 & \text{  otherwise,} 
\end{array} \right. \\
y_{i} &= \left\{ 
\begin{array}{cl}
1 & \text{  if } i \in N \text{ is within source set} \\ 
0 & \text{  otherwise.}
\end{array} \right. \\
\end{aligned}
\end{equation*}
Using this set of variables, we write the \eqref{eq:problemPNI} formulation as follows: 
}

When the arc capacities are stochastic, we are interested in an optimal interdiction plan that minimizes the maximum flow-at-risk. Unlike the expectation criterion used in previous stochastic interdiction models, this approach provides a confidence level for the maximum flow on the arcs that are not interdicted.
Letting $c$ be the mean capacity vector and $Q$ the covariance matrix, the mean-risk network interdiction problem is modeled as
	\begin{align*}
	\min \quad  & c' x + \Omega \sqrt{x'Q x} \\ 
\text{s.t.} \quad	& B y \le x + z,\\ 
(\pni)	\ \ \ \ \ \quad \quad &\alpha^{\prime} z \leq \beta,  \\ 
	&y_s = 1, \ y_t =0, \\  
	&x \in \{0,1\}^A,  y \in \{0,1\}^N,  z \in \{0,1\}^A,  
	\end{align*}
\noindent where $B$ is the node-arc incidence matrix of $G$. Here, $z_a$ is one if arc $a$ is interdicted 
at a cost of $\alpha_a$ and
zero otherwise; and $x_a$ is one if arc $a$ is in the minimum mean-risk cut and zero otherwise. The optimal value of (\pni) is the ``flow-at-risk" for a given interdiction budget $\beta$.
Note that when $\Omega = 0$, (\pni) reduces to the deterministic network interdiction model of \citet{wood1993deterministic}; and, in addition, if $z$ is a vector of zeros, it reduces to the standard $s-t$ min-cut problem. \ins{In a recent paper
\citet{LSS:interdiction} give a scenario-based approach stochastic network interdiction under conditional value-at-risk measure.}
The following example underlines the difference of interdiction solutions between deterministic and mean-risk models with stochastic capacities. 

\begin{example}
Consider the simple network in Figure~\ref{fig:mrni-ex} with two arcs from $s$ to $t$. Arc 1 has mean capacity 1 and 0 variance, whereas arc 2 has mean capacity 0.9 and variance $\sigma^2$. Suppose the budget allows a single-arc interdiction. Then, the deterministic model with $\Omega=0$ would interdict arc 1 with higher mean and leave arc 2 with high variance intact. Consequently, the maximum $s-t$ flow would exceed $0.9+0.5 \sigma$  with probability 0.3085 according to the normal distribution. On the other hand, the mean-risk model with $\Omega > 0.2$,  interdicts arc 2 with lower mean, but high variance ensuring that the maximum $s-t$ flow to be no more than 1.
\end{example}
\begin{figure}[h] 
	\begin{center}
		\includegraphics[width=0.32 \linewidth]{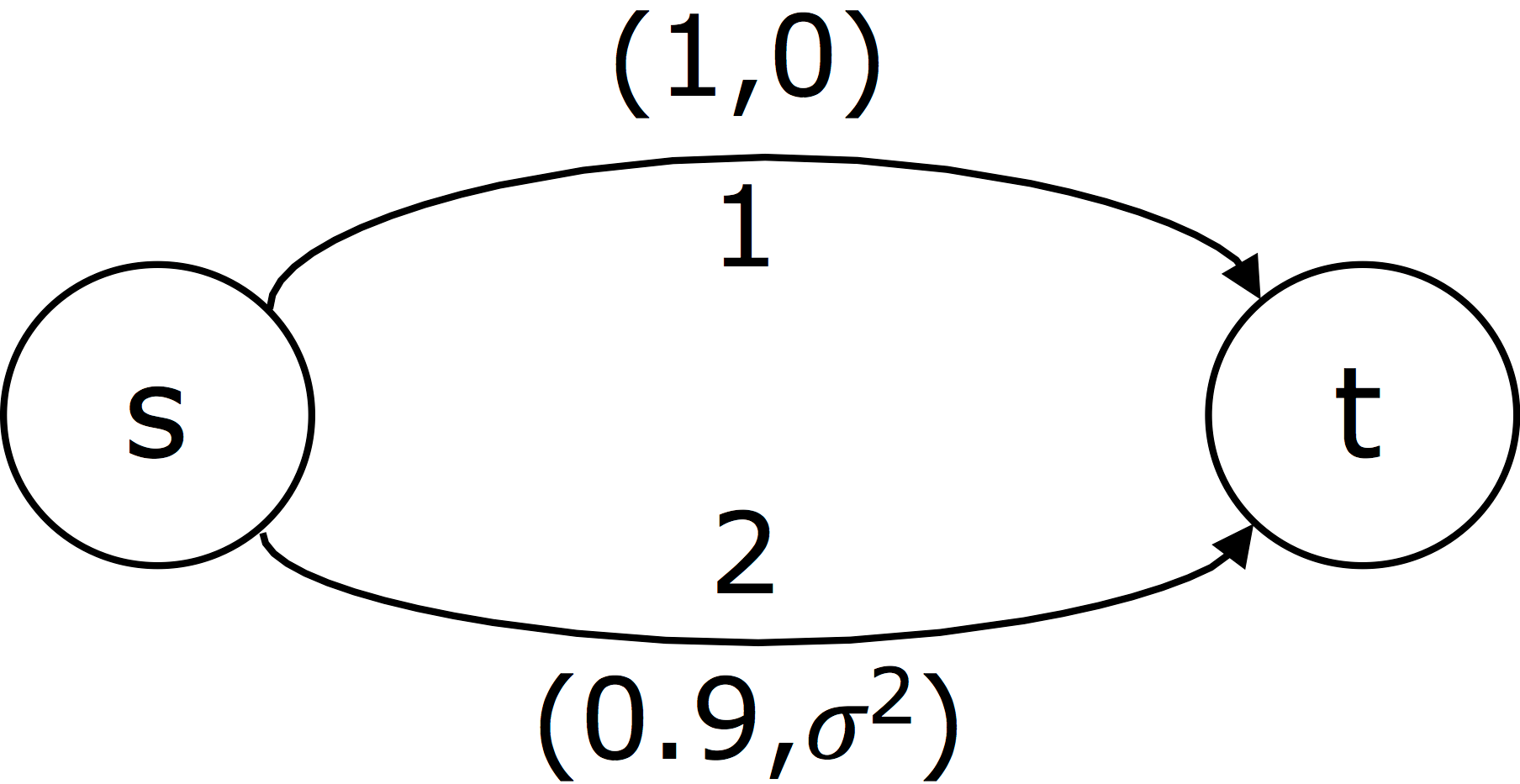}
		\caption{Mean-risk network interdiction.} 
		\label{fig:mrni-ex}
	\end{center} 
\end{figure}

\begin{figure}[h] 
	\begin{center}
		\includegraphics[width=0.65 \linewidth]{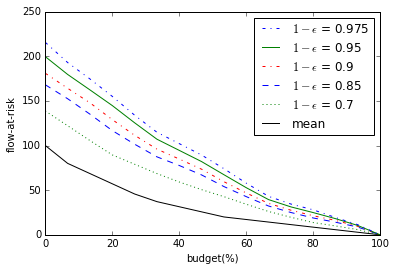}
		\caption{Flow-at-risk vs. interdiction budget for risk aversion levels.} 
		\label{fig:efficientFrontier}
	\end{center} 
\end{figure}

The combinatorial aspect of network interdiction, coupled with correlations, make it extremely challenging to determine the least cost subset of arcs to interdict for a desired confidence level in the maximum flow even for moderate sized networks. 
Yet, understanding the cost and benefit of an interdiction strategy is of critical interest for planning purposes. Toward this end, the proposed approach in the current paper allows one to quickly build efficient frontiers of flow-at-risk vs. interdiction cost, which would, otherwise, be impractical for realistic sizes. Figure \ref{fig:efficientFrontier} shows the flow-at-risk as a function of the interdiction cost for different confidence levels for a $15 \times 15$ grid graph shown in Figure~\ref{fig:structureG}. At 100\% budget the network is interdicted completely allowing no flow.
At lower budget levels the flow-at-risk increases significantly with the higher confidence levels. The vertical axis is scaled so that the deterministic min-cut value 
with $\Omega=0$ at 0\% budget (no interdiction) is 100. The green solid curve corresponding to 95\% confidence level shows that, if no interdiction is performed, the flow on the network is more than 200\% of the deterministic maximum flow with probability 0.05. \ins{The same curve shows} with 40\% interdiction budget \del{, however,} the flow is higher than the deterministic maximum flow \ins{(100)} with probability only 0.05.


\exclude{
	The warm start procedure works well in practice even for the general case where $Q \succ 0$. 
	To show this, we apply the procedure on a problem based on the Deterministic Network Interdiction (DNI) problem introduced by Wood \cite{wood1993deterministic}. 	
	The original problem can be seen as a two stage sequential game over a graph, where an interdictor first removes a fixed number of arcs in order to minimize the maximum flow a smuggler can transport afterwards through it. By using duality, Wood \cite{wood1993deterministic} showed that this problem can be formulated as a single stage optimization problem that minimizes the capacity of the minimum cut. 
	
	We consider a network interdiction problem with uncertain capacities. 
	
	In this section, we apply the warm-start procedure to an extension of the 
	Deterministic Network Interdiction problem (DNI) introduced by \citet{wood1993deterministic}. 
	The original problem can be viewed as a two-stage sequential game over a graph, where an interdictor first removes a fixed number of arcs in order to minimize the maximum flow a smuggler can transport afterwards through it. Using duality, \citet{wood1993deterministic} showed that this problem can be formulated as a single stage optimization problem that minimizes the capacity of the minimum cut. 
	
	We consider an extension of this problem by replacing the linear objective representing the cut capacity 
	with a conic quadratic mean-risk objective to incorporate uncertainty in arc capacities, 
	and refer to this problem as the Probabilistic Network Interdiction problem (PNI). 
	To the best of our knowledge, this is the first study that accounts for arc capacity uncertainty; previous works 
	that introduce uncertainty into the DNI focus primarily on the probability of success of an interdiction attempt  
	\cite{cormican1998stochastic,janjarassuk2008reformulation}.
	Also, our approach to the solution of this problem is different from these studies that applying stochastic programming techniques.
	
	In the remainder of this section, we first describe the problem formulation and the structure of the random instances generated for experiments. 
	Computational results are reported for the two cases where the conic quadratic term in the objective is diagonal and non-diagonal.  
	Finally, we demonstrate how the procedure can help creating an efficient frontier for different scenarios of the problem to aid decision makers. 
	\exclude{
		We first proceed by describing the model formulation, followed by a description of the problem instances. Afterwards, we present results for both the specific case where $Q$ is diagonal and when it is a general positive definite matrix, together with the implementation issues regarding the warm start procedure in each case. 
	}

When there is no correlation between capacities of arcs, $\Sigma$ is a diagonal matrix of variances, i.e.,
$\Sigma = \diag(\sigma_1^2,\cdots,\sigma^2_{m})$. 
The following two subsections consider two types of instances with diagonal and non-diagonal covariance matrices, respectively. 

	Given the speed improvements introduced in Sections \ref{sec:Relax} and \ref{sec:heuristic} for solving the PNI with arc capacity uncertainty, this model can be used to provide insight for decision makers. In this work, we provide the example of a decision maker that must decide the interdiction budget $b$, in the same spirit as \cite{royset2007solving}. However, unlike the previous reference, this decision can be made by considering arc capacity uncertainty. 
	
	To this end, Figure \ref{fig:efficientFrontier} plots how the cut's value-at-risk, which is the objective function in \eqref{eq:probForm}, varies as a function of the interdiction budget. This variation is measured in relative terms, by taking the quotient of the objective value for each budget level with respect to the objective value for the case where $b=0$. In this way, we can compare these curves for different risk aversion levels, which correspond to the weight of the cut's standard deviation in the objective function ($\Omega$).

With the expedited solution process through the use of polymatroid inequalities and the warm-start procedure, 
\eqref{eq:problemPNI} can provide insight for decision makers more efficiently. 
For example, solving multiple instances of \eqref{eq:problemPNI} can benefit a decision maker that must 
decide how much interdiction budget to allocate, 
in an approach similar to the efficient frontier generation discussed by   \cite{royset2007solving}. 
}

\exclude{
	From Figure \ref{fig:efficientFrontier}, we can see that for intermediate values of $\Omega$, where the cuts' expected value and weighted standard deviation are both relevant, the cut's value-at-risk falls more slowly compared to low ($\Omega = 0$) or high ($\Omega = 10$) values of $\Omega$, where the decision maker can focus on minimizing either the expected capacity or its variance, which dominate the objective function in these cases.
	
	Note also that the times required for generating the six curves amounts to 4,552 seconds, which required solving over 100 PNIs for different combinations of $\Omega$ and $b$, which is comparable to the times reported in Janjarassuk and Linderoth for solving one $20 \times 20$ instance with 200 scenarios (see Table 3 in \cite{janjarassuk2008reformulation}). 
}

\subsubsection*{Contributions and outline}
In Section~\ref{sec:qp-ub}, we give a non-convex upper-bounding function for (\mr) that matches 
the mean-risk objective value at its local minima. Then, we describe an
upper-bounding procedure that successively solves quadratic optimization problems instead of conic quadratic optimization. The rationale behind the approach is that algorithms for quadratic optimization with linear constraints scale better than interior point algorithms for conic quadratic optimization. Moreover, simplex algorithms for quadratic optimization can be effectively integrated into branch-and-bound algorithms and other iterative procedures as they allow fast warm-starts. In Section~\ref{sec:comp}, we test the effectiveness of the proposed approach on the network interdiction problem with stochastic capacities and compare it with exact algorithms. 
We conclude in Section~\ref{sec:conclusion} with a few final remarks.

\section{A successive quadratic optimization approach} \label{sec:qp-ub}

In this section, we present a successive quadratic optimization procedure to obtain feasible solutions to (\mr). 
The procedure is based on a reformulation of (\mr) using the perspective function of the convex quadratic term $q(x) = x'Q x$. 
\citet{AG:simplex-qp} introduce
\leqnomode
\begin{align*}
(\po) \ \ \ 
\min
\left\{c^{\prime} x + \frac{\Omega}{2} h(x,t)  + \frac{\Omega}{2} t: x \in X, \ t \ge 0 \right\},   
\end{align*}
\reqnomode
where \ins{$\Omega$ is a positive scalar as before,}
$h: \mathbb{R}^n \times \mathbb{R}_+ \rightarrow \mathbb{R} \cup \{\infty\}$ is the closure of the perspective function of $q$ and is defined as 
\begin{align*}
h(x,t) := 
\begin{cases}
\frac{x' Q x}{t}  & \quad \text{ if } t > 0, \\ 
0 & \quad \text{ if } t=0, \ x'Qx = 0, \\ 
+ \infty & \quad \text{ otherwise. }   
\end{cases}
\end{align*}
As the perspective of a convex function is convex \citep{hiriart2013convex}, $h$ is convex.
\citet{AG:simplex-qp} show the equivalence of (\mr) and (\po) for a polyhedral set $X$. 
Since we are mainly interested in a discrete feasible region, we study (\po) for $X \subseteq \Z^n$.

For $t \ge 0$, it is convenient to define the optimal value function
\begin{align} \label{prob:qp}
f(t) :=\underset{x \in X}{\min }\left\{ g(x,t) := c^{\prime} x + \frac{\Omega}{2} h(x,t) + \frac{\Omega}{2} t \right \} \cdot 
\end{align}
Given $t$, optimization problem \eqref{prob:qp} has a convex quadratic objective function. Let $x(t)$ be a minimizer of  \eqref{prob:qp} with value $f(t)$. Note that $g$ is convex and $f$ is a point-wise minimum of convex functions in $t$ for each choice of $x \in X$, and is, therefore, typically non-convex 
(see \ins{Figure~\ref{fig:fT}}). We show below
that, for any $t \ge 0$, $f(t)$ provides an upper bound on the mean-risk objective value for $x(t)$.

\begin{lemma} \label{lemma:grad}
	$\sqrt{a} \le \frac{1}{2} (a/t + t)$ for all $a,t \ge 0$. 
\end{lemma}
\begin{proof}
	Since $\sqrt{a}$ is concave over $a \ge 0$, it is bounded above by its gradient line:
	\[
	\sqrt{a} \le \sqrt{y} + \frac{1}{2\sqrt{y}} (a-y)
	\]
	at any point $y \ge 0$. Letting $t = \sqrt{y}$ gives the result.
\end{proof}

\begin{proposition} \label{prop:ub}
	For any $t \ge 0$, we have 
	\begin{align*}
	\label{eq:propLocalMin}
	c' x(t) + \Omega \sqrt{{x(t)}' Q x(t)}  
	\le f(t).  
	\end{align*}
\end{proposition}
\begin{proof}
	Applying Lemma~\ref{lemma:grad} with $a = x' Q x \ (\ge 0 \text{ as $Q$ is positive semidefinite})$ gives
	\[
	\sqrt{x'Q x} \le \frac{1}{2} h(x,t) + \frac{t}{2}, \  \forall x \in \R^n, \ \forall t \ge 0. 
	\]
	First multiplying both sides by $\Omega \ge 0$ and then adding $c'x$ shows
	\begin{align*}
	c' x + \Omega \sqrt{x' Q x} &\leq c' x + \frac{\Omega}{2} h(x,t) + \frac{\Omega}{2} t , \ \forall x \in \R^n, \ \forall t \ge 0.   
	\end{align*} 
	The inequality holds, in particular, for $x(t)$ as well.
\end{proof}

\begin{example}
	Consider the mean-risk optimization problem
	\[ \min \bigg \{ x_2 + \sqrt{10x_1^2 + 5 x_2^2} : x \in  X = \{(0,1), (1,0)\} \subseteq \R^2 \bigg \}
	\]
	with two feasible points.
	Figure \ref{fig:fT} illustrates the optimal value function $f$.
	The curves in red and green show $g((1,0), t)$ and $g((0,1), t)$, respectively, 
	and  $f(t) = \min \{g((1,0), t), g((0,1), t)\}$ is shown with a dotted line.   
	As the red and green curves intersect at $t = 2.5$, $x(t)$ is $(0,1)$ for $t \leq 2.5$, and $(1,0)$ for $t \ge 2.5$. 
	
	In this example, $f$ has two local minima: $1+\sqrt{5}$ attained at $t=\sqrt{5}$ and $\sqrt{10}$ at $t=\sqrt{10}$.  Observe that the upper bound $f(t)$ matches the mean-risk objective at these local minima:
	\[c' x({t}) + \Omega \sqrt{x({t})' Q x({t})}  = f({t}), \ \ {t} \in  \{\sqrt{5}, \sqrt{10}\}. \]
	The black step function shows the mean-risk values for the two feasible solutions of $X$.  
	It turns out the upper bound $f(t)$ is tight, in general, for any local minima (Proposition~\ref{prop:localMinima}).
	
	In order to contrast the convex and discrete cases, we show with
	solid blue curve the lower bound $\hat{f}$ of $f$, where $\hat f(t) = \min \{g(x,t): x \in \hat X \}$ and 
	$\hat{X} := \{(x_1, x_2) \in \R^2_+: x_1 + x_2 = 1 \}$ 	
	is the convex relaxation of $X$.
	Let $\hat{x}(t)$ be the solution of this convex problem.
	Then $\hat{f}(t)$ provides an upper bound on $c' \hat{x}(t)+ \Omega \sqrt{{\hat{x}(t)}' Q \hat{x}(t)} $ (graph shown in dotted blue curve)
	at any $t \ge 0$, and the bound is tight at $t = \sqrt{25 / 7}$, where the minimum of $\hat{f}(t)$ is attained.
	\begin{figure}[h]
		\begin{center}
			\includegraphics[scale=0.65]{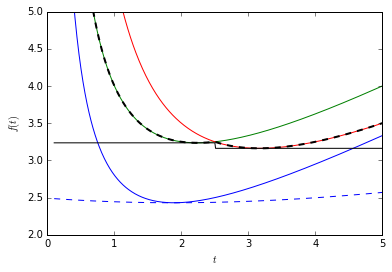}
			\caption{The value function $f$ with two discrete feasible points.}
			\label{fig:fT}
		\end{center}
	\end{figure}
\end{example}

Although, in general, $f(t)$ provides an upper bound, the next proposition shows that the mean-risk objective and $f$ match at local minima of $f$. 

\begin{proposition} \label{prop:localMinima}
	If $f$ has a local minimum at $\bar{t} > 0$, then we have
	\begin{align}
	\label{eq:propLocalMin}
	c' x(\bar{t}) + \Omega \sqrt{{x(\bar{t})}' Q x(\bar{t})}  
	= f(\bar{t}).  
	\end{align}
\end{proposition}

\begin{proof}
	Since $f$ is the point-wise minimum of differentiable convex functions, it is differentiable at its local minima in the interior of its domain ($t > 0$). Then, its vanishing derivative at $\bar t$
	\[
	f'(\bar t) = -\frac{x(\bar t)' Q x(\bar t)}{\bar t^2} + 1 = 0
	\] 
	implies $\bar t = \sqrt{x(\bar t)' Q x(\bar t)}$. Plugging this expression into $f(\bar{t})$ gives the result.
\end{proof}

\exclude{
	\begin{proof}
		Since $\bar{t}$ is a local minimum point, there exists $\delta > 0$ such that 
		$f(\bar{t}) \leq f(t) \ \forall \, t \in [\bar{t}-\delta, \bar{t}+\delta]$. Equivalently,
		\begin{align*}
		c' x(\bar{t}) + \frac{\Omega}{2\bar{t}} {x(\bar{t})}' Q x(\bar{t}) + \frac{\Omega}{2}\bar{t}  
		& \leq {\min_{x \in S} } \left\{ c' x + \frac{\Omega}{2(\bar{t} + \epsilon)} x' Q x + \frac{\Omega(\bar{t} + \epsilon)}{2} \right\}, \ \forall  \epsilon \in [-\delta, \delta]  \\ 
		& \leq c' x(\bar{t}) + \frac{\Omega}{2(\bar{t} + \epsilon)} {x(\bar{t})}' Q x(\bar{t}) + \frac{\Omega(\bar{t} + \epsilon)}{2} , \ \forall  \epsilon \in [-\delta, \delta] 
		\end{align*}
		Simplifying, the last inequality implies that
		\[
		0  \geq \frac{\epsilon}{2\bar{t}(\bar{t} + \epsilon)} {x(\bar{t})} ' Q x(\bar{t}) - \frac{\epsilon}{2}, \ \forall  \epsilon \in [-\delta, \delta].
		\]
		
		Multiplying the inequality by $2\bar{t}(\bar{t} + \epsilon)/\epsilon$ for 
		$\epsilon \neq 0$ leads to the following two inequalities depending on the sign of $\epsilon$: 
		\begin{align*}
		0 & \geq  {x(\bar{t})} ' Q x(\bar{t}) - \bar{t} (\bar{t} + \epsilon), \ \forall \epsilon \in (0, \delta], \\ 
		0 & \leq {x(\bar{t})} ' Q x(\bar{t}) -  \bar{t} (\bar{t} + \epsilon),  \ \forall \epsilon \in [-\delta, 0).   
		\end{align*}
		Taking the limit as $\epsilon \rightarrow 0+$ and $\epsilon \rightarrow 0-$, respectively, 
		we conclude that 
		\[\bar{t} = \sqrt{x(\bar{t}) ' Q x(\bar{t})}. \]
		Plugging in this expression for $\bar{t}$ into $f(\bar{t})$ establishes the result.
	\end{proof}
}

Finally, we show that problems (\mr) and (\po) are equivalent. In other words, the best upper bound matches
the optimal value of the mean-risk problem, which provides an alternative way for solving (\mr).

\begin{proposition} \label{prop:equiv}
	Problems (\mr) and (\po) are equivalent; that is,
	\begin{align*}
	{\min} \left\{c^{\prime} x + \Omega \sqrt{x^{\prime} Q x}: x \in X \right\} 
	= \min \{ f(t): t \ge 0 \} \cdot
	\end{align*}
\end{proposition}

\begin{proof}
	Let $t^*$ be optimal for $\min \{ f(t): t \ge 0 \}$. By Proposition~\ref{prop:ub}
	\[
	f(t^*) \ge c'x(t^*) + \Omega \sqrt{x(t^*)'Qx(t^*)} \ge {\min} \left\{c^{\prime} x + \Omega \sqrt{x^{\prime} Q x}: x \in X \right\}.
	\]	
	The other direction follows from the observation 
	\begin{align*}
	\min_{x \in X} \left\{c'x + \Omega \sqrt{x'Qx} \right\} & = \min_{x \in X, \; \ins{t \ge 0}} \left\{ c'x + \frac{\Omega}{2} \text{\rep{h}{g}}(x,t) + \frac{\Omega}{2} t : t = \sqrt{x'Qx} \right\} \\
	& \geq  \min_{x \in X, \; t \ge 0} \left\{ c'x + \frac{\Omega}{2} \text{\rep{h}{g}}(x,t) + \frac{\Omega}{2} t  \right\} =  \min_{t \ge 0} \{ f(t) \} \cdot
	\end{align*} 
\end{proof}

The one-dimensional upper-bounding function $f$ above suggests a local search
algorithm that utilizes quadratic optimization to evaluate the function at any $t \ge 0$:
\[ f(t) = \min_{x \in X} \left \{ g(x,t) := c'x + \frac{\Omega}{2t} x'Qx + \frac{\Omega}{2} t \right \}\]
and avoids the solution of a conic quadratic optimization problem directly.

Algorithm~\ref{alg:bisection} describes a simple binary search method that halves the uncertainty interval $[t_{min}, t_{max}]$, initiated as $t_{min} = 0$ and $t_{max} = \sqrt{\bar x'Q \bar x}$, where $\bar x$ is an optimal solution to (\mr) with $\Omega=0$. The algorithm is terminated either when a local minimum of $f$ is reached or the gap between the upper bound $f(t)$ and $c' x(t) + \Omega \sqrt{x(t)'Q x(t)}$ is small enough. For the computations  in Section ~\ref{sec:comp} we use 1\% gap as the stopping condition.

\ignore{
	\begin{algorithm}
		\DontPrintSemicolon
		\SetAlgoNoEnd\SetAlgoNoLine%
		{\texttt{0. \hskip -3mm Initialize}} \\ \texttt{Set} $\texttt{lb} = 0$, $\texttt{ub} = u$, $t = (\texttt{lb}+\texttt{ub})/2$ \\
		\hrule
		\bigskip
		\texttt{1. \hskip -3mm Update \& Terminate} \\
		\quad \texttt{Compute $f(t)$ and $x(t)$} \; 
		\quad \uIf{$(f(t) - (c'x(t) + \Omega \sqrt{x(t)'Qx(t)} ))/f(t) \leq \Delta$ or \texttt{ub}$-\texttt{lb} \leq \Delta$}{ \texttt{Return $x(t)$.}}
		\quad \ElseIf{ $\frac{g}{\partial t}(x(t), t) \leq -\epsilon$}{$\texttt{ub} = t$ \; \texttt{Repeat step 1.}} 
		\quad \ElseIf{ $\frac{g}{\partial t}(x(t), t) \geq \epsilon$}{$\texttt{lb} = t$ \; \texttt{Repeat step 1.}} 
		\quad \Else{ \texttt{Return $x(t)$}}
		\caption{Local search procedure.}
		\label{alg:WS}
	\end{algorithm}

}

\begin{algorithm}[h]
	\caption{Binary local search.}
	\label{alg:bisection}
	\begin{algorithmic}[1]
		\renewcommand{\algorithmicrequire}{\textbf{Input:}}
		\renewcommand{\algorithmicensure}{\textbf{Output:}}
		\Require $X \subseteq \Z^n; Q\text{ p.s.d. matrix; }c\text{ cost vector; } \Omega>0$
		\Ensure Local optimal solution $x$
		\State \textbf{Initialize }$t_{\min}$ and $t_{\max}$ 
		\State $\hat{z}\leftarrow \infty$ \Comment{best objective value found}
		\Repeat
		\State $t\leftarrow \frac{t_{\min}+t_{\max}}{2}$
		\State $x(t)\leftarrow \argmin\left\{c'x+\frac{\Omega}{2t}x'Qx+\frac{\Omega}{2}t: x \in X\right\}$ \ \ 
		\If{$\frac{\partial g}{\partial t}(x(t), t) \leq -\epsilon$}\label{line:iBisection10} 
		\State $t_{\min}\leftarrow t$\label{line:iBisection11}
		\ElsIf{ $\frac{\partial g}{\partial t}(x(t), t) \geq \epsilon$} \label{line:iBisection20}
		\State $t_{\max}\leftarrow t$\label{line:iBisection21}
		\Else \label{line:iBisection20}
		\State {return $x(t)$}
		\EndIf \label{line:iBisection3}
		\Until stopping condition is met \label{line:stoppingCriterion2}
		\State \Return $\hat{x}$
	\end{algorithmic}
\end{algorithm}

\subsubsection*{The uncorrelated case over binaries}
The reformulation (\po) simplifies significantly for the special case of independent random variables over binaries. In the absence of correlations, the covariance matrix reduces to a diagonal matrix $Q = diag(q)$, where $q$ is the vector of variances. For 
\begin{align*} 
(\dmr) \ \ \ \min \left\{c' x + \Omega \sqrt{q' x} :  x \in X \subseteq \mathbb{B}^n \right\} 
\end{align*}
the upper bounding problem simplifies to
\begin{align} \label{prob:ub-d}
f(t) = \min \left\{c' x + \frac{\Omega}{t} q' x  + \frac{\Omega}{2} t:  x \in X \subseteq \mathbb{B}^n \right\},
\end{align}
which is a binary linear optimization problem for fixed $t$. Thus, $f$ can be evaluated fast for \del{many} \ins{linear} combinatorial optimization problems, \ins{such as the  minimum spanning tree problem, shortest path problem, assignment problem, minimum cut problem \cite{book:S:co},} for which there exist polynomial-time algorithms. Even when the evaluation problem 
\eqref{prob:ub-d} is \NP-hard, simplex-based branch-and-bound algorithms equipped with warm-starts perform much faster than conic quadratic mean-risk minimization as demonstrated in the next section.

\section{Computational Experiments}
\label{sec:comp}

In this section we report on computational experiments conducted to test the effectiveness of the
proposed successive quadratic optimization approach on the network interdiction problem with stochastic capacities. 
We compare the solution quality and the computation time with exact algorithms.

All experiments are carried out using CPLEX 12.6.2 solver on a workstation 
with a 3.60 GHz Intel R Xeon R CPU E5-1650 and 32 GB main memory and with a single thread. 
Default CPLEX settings are used with few exceptions: dynamic search and presolver are disabled to utilize the user cut callback; the branch-and-bound nodes are solved using linear outer approximation for faster enumeration; and the time limit is set to one hour.  

\subsubsection*{Problem instances} 
\label{subsec:inst}
We generate \del{network interdiction} instances \ins{of the mean-risk network interdiction problem (\pni)}
on grid graphs similar to the ones used in
\citet{cormican1998stochastic,janjarassuk2008reformulation}. 
Let \textit{ $p \times q$ grid} be the graph with $p$ columns and $q$ rows of grid nodes in addition to a source and a sink node (see Figure \ref{fig:structureG}). 
The source and sink nodes are connected to all the nodes in the first and last column, respectively. 
The arcs incident to source or sink have infinite capacity and are not interdictable.  
The arcs between adjacent columns are always directed toward the sink, 
and the arcs connecting two nodes within 
the same column are directed either upward or downward with equal probability. 

\begin{figure}[h]
\begin{center}
\includegraphics[width=0.55 \linewidth]{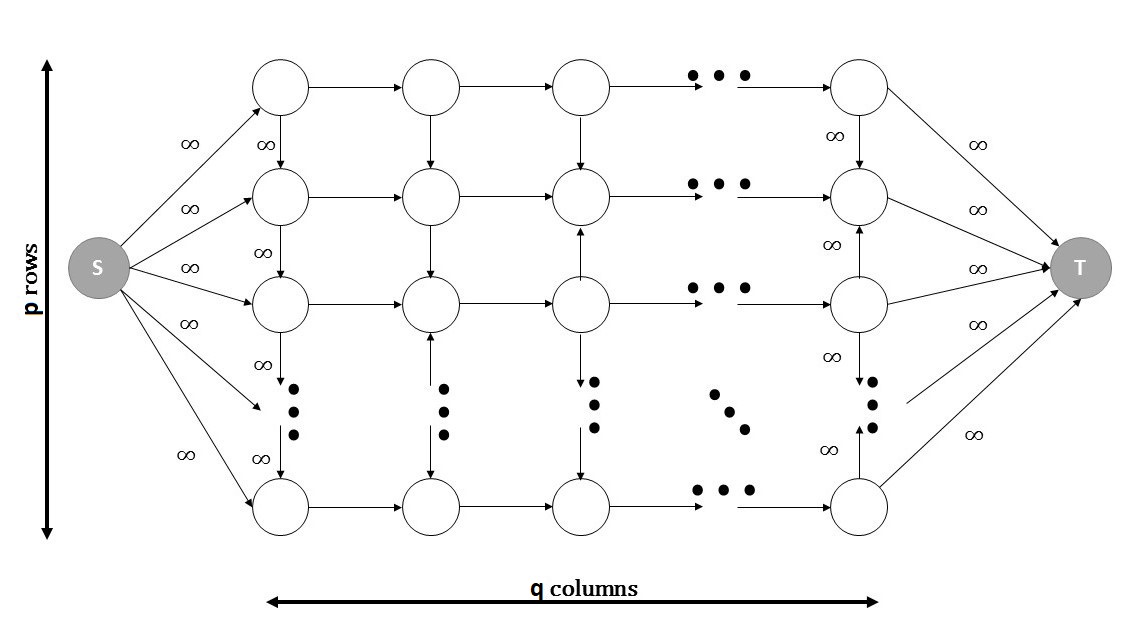}
\caption{\del{Structure of }$p \times q$ grid graph.}
\label{fig:structureG}
\end{center}
\end{figure}

We generate two types of data: uncorrelated and correlated.
For each arc $a \in A$, the mean capacity $c_a$ 
and its standard deviation $\sigma_a$ are independently drawn from the integral uniform $[1, 10]$, 
and the interdiction cost $\alpha_a$ is drawn from the integral uniform $[1, 3]$. 
For the correlated case, the covariance matrix is constructed via a factor model: $Q = \diag(\sigma_1^2,\cdots,\sigma^2_{|A|})+ E F E' $, 
where $F$ is an $m \times m$ factor covariance matrix and $E$ is the exposure matrix of the arcs to the factors. 
$F$ is computed as $F = HH'$, where each $H_{ij}$ is drawn from uniform $[-100/pq, 100/pq]$, and each $E_{ij}$ 
from uniform $[0, 0.1]$ with probability 0.2 and set to 0 with probability 0.8. 
The interdiction budget $\beta$ is set to $\lceil \frac{Y}{2} \rceil$, 
and the risk averseness parameter $\Omega$ is set to $\Phi^{-1}(1-\epsilon)$, where $\Phi$ is the c.d.f. of the standard normal distribution. 
Five instances are generated for each combination of graph sizes $p \times q$ : $10\times 10$, $20 \times 20$, $30 \times 30$
and confidence levels $1-\epsilon$: 0.9, 0.95, 0.975. The data set is available
for download at \texttt{http://ieor.berkeley.edu/$\sim$atamturk/data/prob.interdiction} .

\ins{For completeness, we state the corresponding perspective optimization for (\pni):
\begin{align*}
\min \quad  & c' x + \Omega x'Q x/2t + \Omega t/2 \\ 
\text{s.t.} \quad	& B y \le x + z,\\ 
(\text{PO}-\pni)	\ \ \ \ \ \quad \quad &\alpha^{\prime} z \leq \beta,  \\ 
&y_s = 1, \ y_t =0, \\  
&x \in \{0,1\}^A,  y \in \{0,1\}^N,  z \in \{0,1\}^A,  t \in \R_+.
\end{align*}
}

\subsubsection*{Computations} 
Table~\ref{tab:ws} summarizes the performance of the successive quadratic optimization approach on the network interdiction instances. 
We present the number of iterations, the computation time in seconds, and the percentage optimality gap for the solutions, separately for the uncorrelated and correlated instances.  Each row represents the average over five instances for varying grid sizes and confidence levels. One sees in the table that  only a few number of iterations are required to obtain solutions within about 1\% of optimality for both the correlated and uncorrelated instances.
While the solution times for the correlated case are higher, even the largest instances are solved under 20 seconds on average. 
The computation time increases with the size of the grids, but is not affected by the confidence level $1-\epsilon$.

\begin{table}[h]
	\caption{Performance of the binary local search.}
\label{tab:ws}
	\footnotesize
	\centering
	\setlength{\tabcolsep}{3pt} 
	\begin{tabular}{c|c|rrr|rrr}
		\hline \hline
		\multicolumn{2}{c|}{} & \multicolumn{3}{c|}{ Uncorrelated } & \multicolumn{3}{c}{Correlated}  \\
		\hline
		$p \times q$ & $1-\epsilon$ & iter & time & gap & iter & time & gap \\ 
		\hline
		\multirow{3}{*} {$10 \times 10$} 
		& 0.9  & 2.8 & 0.04 & 1.47 & 3.0 & 0.10 & 0.68 \\
		& 0.95 & 3.4 & 0.05 & 0.28 & 3.0 & 0.11 & 0.74 \\
		& 0.975 & 2.8 & 0.05 & 0.00 & 3.0 & 0.09 & 0.73 \\
		\hline
		\multirow{3}{*} {$20 \times 20$} 
		& 0.9  & 3.0 & 0.49 & 1.54 & 4.0 & 2.74 & 0.35 \\
		& 0.95 & 2.8 & 0.36 & 1.06 & 4.0 & 2.68 & 0.44 \\
		& 0.975 & 2.8 & 0.45 & 1.07 & 4.0 & 3.21 & 5.86 \\
		
		\hline
		\multirow{3}{*}{$30 \times 30$} 
		& 0.9  & 3.0 & 2.24 & 1.67 & 5.0 & 16.00 & 0.46 \\
		& 0.95 & 3.0 & 2.54 & 1.26 & 5.8 & 16.87 & 0.15 \\
		& 0.975 & 3.0 & 2.58 & 1.19 & 5.8 & 18.57 & 0.20 \\
		\hline
		\multicolumn{2}{c|}{\textbf{avg} \rule{0pt}{10pt}} 
		& $\BD{2.96}$ & $\BD{0.98}$ & $\BD{1.06}$ & $\BD{4.18}$ & $\BD{6.71}$ & $\BD{1.07}$ \\ 
		\hline \hline
	\end{tabular}
\end{table}

\begin{table}[t]
	\footnotesize
	\centering
	\caption{Performance of b\&b and b\&c algorithms.}
	\label{tb:diag5}
	\setlength{\tabcolsep}{2pt} 
	\begin{tabular}{c|c|rrrrr|rrrrrr}
		\hline \hline
			\multicolumn{2}{c}{}  \rule{0pt}{10pt} 	& \multicolumn{11}{|c}{ Uncorrelated instances }  \\ \hline
		\multicolumn{2}{c}{} & \multicolumn{5}{|c|}{ Cplex } & \multicolumn{6}{c}{Cplex $+$ cuts} \\ 
		\hline
		$p \times q$ & $1-\epsilon$ & rgap & stime & time & egap (\#) & nodes  & cuts & rgap  & stime & time & egap (\#) & nodes  \\ 
		\hline
		\multirow{3}{*} {$10 \times 10$} 
		& 0.9   & 15.1 & 0 & 1 & 0.0\phantom{(0)} & 457   & 101 & 5.1 & 0 & 3 & 0.0\phantom{(0)} & 11   \\
		& 0.95  & 17.0 & 1 & 1 & 0.0\phantom{(0)} & 1,190 & 127 & 5.6 & 2 & 4 & 0.0\phantom{(0)} & 75  \\ 
		& 0.975 & 17.9 & 1 & 2 & 0.0\phantom{(0)} & 1,194 & 137 & 6.0 & 3 & 4 & 0.0\phantom{(0)} & 73  \\ 
		\hline
		\multirow{3}{*} {$20 \times 20$} 
		& 0.9   & 17.9 &  66 & 169 & 0.0\phantom{(0)}  & 23,093  & 463 & 10.2 & 23 &  44 & 0.0\phantom{(0)} &    602  \\ 
		& 0.95  & 20.0 & 469 & 676 & 0.0\phantom{(0)}  & 56,937  & 579 & 11.4 & 48 & 102 & 0.0\phantom{(0)} &  4,850   \\ 
		& 0.975 & 21.6 & 404 & 1,365 & 0.5(1)          & 91,786  & 621 & 12.5 & 79 & 262 & 0.0\phantom{(0)} & 16,883  \\ 
		\hline
		\multirow{3}{*}{$30 \times 30$} 
		& 0.9   & 19.1 & 2,338 & 3,258 &  4.6(4) & 65,475  & 680 & 12.6 & 666 &   838 & 0.0\phantom{(0)} & 11,171  \\ 
		& 0.95  & 21.3 & 3,315 & 3,600 & 10.3(5) & 61,754 & 752 & 14.3 & 850 & 1,313 & 0.0\phantom{(0)} & 22,420  \\ 
		& 0.975 & 23.1 & 3,535 & 3,600 & 15.3(5) & 67,951 & 767 & 15.9 & 1,973 & 2,315 & 1.6(2)           & 35,407  \\ 
		\hline
		\multicolumn{2}{c|}{\textbf{avg}  \rule{0pt}{10pt} } 
		& $\BD{19.2}$ & $\BD{1,125}$ & $\BD{1,408}$ & $\BD{3.4(15)}$ & $\BD{41,093}$  & $\BD{470}$ & $\BD{10.4}$  & $\BD{404}$ &  $\BD{543}$  & $\BD{0.2(2)}$ & $\BD{10,166}$  \\ 
		\hline
					\multicolumn{2}{c}{}  \rule{0pt}{10pt} 	& \multicolumn{11}{|c}{ Correlated instances }  \\ \hline
		\multicolumn{2}{c}{} & \multicolumn{5}{|c|}{ Cplex } & \multicolumn{6}{c}{Cplex $+$ cuts} \\ 
		\hline
		$p \times q$ & $1-\epsilon$ & rgap  & stime & time & egap (\#) & nodes & cuts & rgap  & stime & time & egap (\#) & nodes  \\ 
		\hline
		\multirow{3}{*} {$10 \times 10$} 
		& 0.9   & 10.5 & 2 & 4 & 0.0\phantom{(0)} & 268   & 114 &  5.8 & 4 & 7 & 0.0\phantom{(0)} & 14   \\ 
		& 0.95  & 14.5 & 1 & 2 & 0.0\phantom{(0)} & 727   & 126 &  8.0 & 2 & 6 & 0.0\phantom{(0)} & 44   \\ 
		& 0.975 & 16.2 & 2 & 2 & 0.0\phantom{(0)} & 1,105 & 120 & 10.3 & 2 & 5 & 0.0\phantom{(0)} & 67   \\ 
		\hline
		\multirow{3}{*} {$20 \times 20$} 
		& 0.9   & 15.0 & 49  & 92  & 0.0\phantom{(0)} & 11,783 & 341 & 12.0 & 26 &  31 & 0.0\phantom{(0)} & 1,199   \\ 
		& 0.95  & 16.9 & 75  & 314 & 0.0\phantom{(0)} & 30,536 & 400 & 13.8 & 48 &  81 & 0.0\phantom{(0)} & 3,567   \\ 
		& 0.975 & 18.2 & 802 & 615 & 0.0\phantom{(0)} & 66,759 & 420 & 15.1 & 66 & 129 & 0.0\phantom{(0)} & 6,911  \\ 
		\hline
		\multirow{3}{*}{$30 \times 30$} 
		& 0.9   & 12.1 &  427  & 873   & 0.0\phantom{(0)} & 21,748 & 343 &  9.3 & 130 & 246 & 0.0\phantom{(0)} &  4,325   \\ 
		& 0.95  & 13.3 &  527  & 1,436 & 0.0\phantom{(0)} & 37,448 & 420 & 10.3 & 249 & 295 & 0.0\phantom{(0)} &  4,559   \\ 
		& 0.975 & 13.8 & 1,776 & 2,465 & 0.4(1)           & 59,202 & 529 & 10.8 & 673 & 810 & 0.0\phantom{(0)} & 12,093   \\ 
		\hline
		\multicolumn{2}{c|}{\textbf{avg}  \rule{0pt}{10pt} } 
		&$\BD{14.5}$ & $\BD{386}$ & $\BD{666}$ & $\BD{0.1(1)}$ & $\BD{25,509}$  &  $\BD{313}$ & $\BD{10.6}$ & $\BD{133}$ & $\BD{179}$ & $\BD{0.0\phantom{0)}}$ & $\BD{3,642}$ \\ 
		\hline \hline
	\end{tabular}
\end{table}

\ignore{
\begin{table}[h]
	\footnotesize
	\centering
	\caption{Computations with the correlated case.}
	\label{tb:corr6}
	\setlength{\tabcolsep}{2pt} 
	\begin{tabular}{c|c|rrrrr|rrrrrr}
		\hline \hline
		\multicolumn{2}{c}{} & \multicolumn{5}{|c|}{ Cplex } & \multicolumn{6}{c}{Cplex $+$ cuts} \\ 
		\hline
		$p \times q$ & $1-\epsilon$ & rgap  & stime & time & egap (\#) & nodes & cuts & rgap  & stime & time & egap (\#) & nodes  \\ 
		\hline
		\multirow{3}{*} {$10 \times 10$} 
		& 0.9   & 10.5 & 2 & 4 & 0.0\phantom{(0)} & 268   & 114 &  5.8 & 4 & 7 & 0.0\phantom{(0)} & 14   \\ 
		& 0.95  & 14.5 & 1 & 2 & 0.0\phantom{(0)} & 727   & 126 &  8.0 & 2 & 6 & 0.0\phantom{(0)} & 44   \\ 
		& 0.975 & 16.2 & 2 & 2 & 0.0\phantom{(0)} & 1,105 & 120 & 10.3 & 2 & 5 & 0.0\phantom{(0)} & 67   \\ 
		\hline
		\multirow{3}{*} {$20 \times 20$} 
		& 0.9   & 15.0 & 49  & 92  & 0.0\phantom{(0)} & 11,783 & 341 & 12.0 & 26 &  31 & 0.0\phantom{(0)} & 1,199   \\ 
		& 0.95  & 16.9 & 75  & 314 & 0.0\phantom{(0)} & 30,536 & 400 & 13.8 & 48 &  81 & 0.0\phantom{(0)} & 3,567   \\ 
		& 0.975 & 18.2 & 802 & 615 & 0.0\phantom{(0)} & 66,759 & 420 & 15.1 & 66 & 129 & 0.0\phantom{(0)} & 6,911  \\ 
		\hline
		\multirow{3}{*}{$30 \times 30$} 
		& 0.9   & 12.1 &  427  & 873   & 0.0\phantom{(0)} & 21,748 & 343 &  9.3 & 130 & 246 & 0.0\phantom{(0)} &  4,325   \\ 
		& 0.95  & 13.3 &  527  & 1,436 & 0.0\phantom{(0)} & 37,448 & 420 & 10.3 & 249 & 295 & 0.0\phantom{(0)} &  4,559   \\ 
		& 0.975 & 13.8 & 1,776 & 2,465 & 0.4(1)           & 59,202 & 529 & 10.8 & 673 & 810 & 0.0\phantom{(0)} & 12,093   \\ 
		\hline
		\multicolumn{2}{c|}{\textbf{avg}  \rule{0pt}{10pt} } 
		&$\BD{14.5}$ & $\BD{386}$ & $\BD{666}$ & $\BD{0.1(1)}$ & $\BD{25,509}$  &  $\BD{313}$ & $\BD{10.6}$ & $\BD{133}$ & $\BD{179}$ & $\BD{0.0\phantom{0)}}$ & $\BD{3,642}$ \\ 
		\hline \hline
	\end{tabular}
\end{table}

}

The optimal/best known objective values used for computing the optimality gaps in Table~\ref{tab:ws} are obtained with the CPLEX \del{exact} branch-and-bound algorithm. 
To provide a comparison with the successive quadratic optimization procedure, we summarize the performance for the exact algorithm in Table~\ref{tb:diag5}, for the uncorrelated and correlated instances, respectively. 
In each column, we report the percentage integrality gap at the root node (rgap), 
the time spent until the best feasible solution is obtained (stime),
the total solution time in CPU seconds (time), 
the percentage gap between the best upper bound and the lower bound at termination (egap), 
and the number of nodes explored (nodes).  
If the time limit is reached before proving optimality, the number of instances unsolved (\#) is shown next to egap. Each row of the tables represents the average for five instances. 

\ignore{
\newpage 

$\left . \right .$

\newpage
}

Observe that the solution times with the CPLEX branch-and-bound algorithm are much larger compared to the successive quadratic optimization approach: 1,408 secs. vs. 1 sec. for the uncorrelated instances and 666 secs. vs. 7 secs. for the correlated instances. The difference in the performance is especially striking for the 30 $\times$ 30 instances, of which half are not solved to optimality within the time limit. Many of these unsolved instances are terminated with large optimality gaps (egap).

\ins{In order to strengthen the convex relaxation of 0-1 problems with a mean-risk objective, one can utilize the polymatroid inequalities \cite{AG:poly}. Polymatroid inequalities exploit the submodularity of the mean-risk objective for the diagonal case. They are extended for the (non-digonal) correlated case as well as for mixed 0-1 problems in \cite{atamturk2008polymatroids}. }   
\del{In order} To improve the performance of the exact algorithm, we also test it by adding \ins{the} polymatroid cuts.
\del{ for mean-risk 0-1 optimization problems}  
It is clear in Table~\ref{tb:diag5} that the polymatroid cuts have a very positive impact on the exact algorithm. The root gaps are reduced significantly with the addition of the polymatroid cuts.
Whereas 16 of the instances are unsolved within the time limit with default CPLEX, all but two instances are solved to optimality when adding the cuts. Nevertheless, the solution times even with the cutting planes are much larger compared to the successive quadratic optimization approach: 543 secs. vs. 1 sec. for the uncorrelated case 
and 179 secs. vs. 7 secs. for the correlated case.


\ignore{
The optimal values used for computing the optimality gaps in Table~\ref{tab:ws} are obtained with an exact branch-and-cut algorithm that employs polymatroid cuts for mean-risk 0-1 optimization problems \cite{AG:poly,atamturk2008polymatroids}. To provide a comparison with the
proposed successive quadratic optimization procedure, we summarize the performance for the exact algorithm in Tables~\ref{tb:diag5}~and~\ref{tb:corr6}, for the uncorrelated and correlated instances, respectively. In these tables
we present the results for the
default CPLEX branch-and-bound algorithm and a branch-and-cut algorithm with polymatroid inequalities.
In each column, we report the percentage integrality gap at the root node (rgap), 
the time spent until the best feasible solution is obtained (stime),
the total solution time in CPU seconds (time), 
the percentage gap between the best upper bound and the lower bound at termination (egap), 
and the number of nodes explored (nodes).  
If the time limit is reached before proving optimality, the number of instances unsolved (\#) is shown next to egap. 
For the branch-and-cut algorithm, the number of polymatroid cuts added to the formulation (cuts) is reported as well. 
Each row of the tables represents the average for five instances. 

It is clear in Tables~\ref{tb:diag5}~and~\ref{tb:corr6} that the polymatroid cuts have a very positive impact on the exact algorithm. The root gaps are reduced significantly with the addition of the polymatroid cuts.
Whereas 16 of the instances are not solved within the time limit with default CPLEX, all but two instances are solved to optimality when adding the polymatroid cuts. Nevertheless, the solution times even with the cutting planes are much larger compared to the successive quadratic optimization approach: 543 secs. vs. 1 sec. for the uncorrelated case 
and 179 secs. vs. 7 secs. for the correlated case. Note that the 
correlated instances have lower integrality gap, and hence are solved faster with the exact search algorithms. 

}

Branch-and-bound and branch-and-cut algorithms spend a significant amount of solution time to prove optimality rather than finding feasible solutions. Therefore, for a fairer comparison, it is also of interest to check the time to the best feasible solution, which are reported under the column stime in Table~\ref{tb:diag5}. The average time to the best solution is 1,125 and 386 seconds for the branch-and-bound algorithm and 404 and 133 seconds for the branch-and-cut algorithm for the uncorrelated and correlated cases, respectively.  
Figure~\ref{fig:pp} presents the progress of the incumbent solution over time
for one of the $30 \times 30$ instances. 
The vertical axis shows the distance to the optimal value (100\%). The binary search algorithm finds a solution within 3\% of the optimal under 3 seconds. It takes 1,654 seconds for the \del{CPLEX} default branch-and-bound algorithm and 338 seconds for the branch-and-cut algorithm to find a solution at least as good.

\ignore{
In Figure~\ref{fig:pp}, we present the performance profile of time to the best solution for the algorithms considered. An $(x,y)$ point in the plots shows the number of instances $y$ that required no longer than $x$ seconds to find the best solution.
49 out of 90 instances are solved to optimality by the binary local search algorithm. The average optimality gap for the remaining 41 instances is 2.33\%. Also observe that 74 instances are solved to optimality by the default CPLEX branch-and-bound algorithm within the one hour time limit. 
Better feasible solutions are found faster by the branch-and-cut algorithm. 
}

\begin{figure}[h!]
	\begin{center}
		\includegraphics[width=0.65 \linewidth]{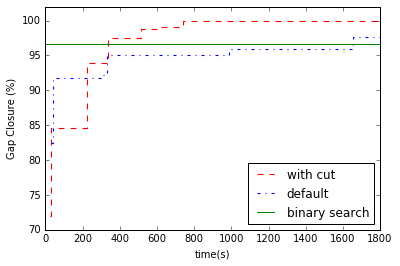}
		\caption{Performance profile of the algorithms.}
		\label{fig:pp}
	\end{center}
\end{figure}

\ins{
The next set of experiments are done to test the impact of the budget constraint on the performance of the algorithms. For these experiments, the instances with $1-\epsilon = 0.95$ and grid size $20 \times 20$ are solved with varying levels of budgets.  Specifically, the budget parameter 
$\beta$ is set to $\frac{\bar{\alpha}Y}{\eta}$ for $\eta \in \{2, 4, 6, 8, 10, 20\}$, 
where $\bar{\alpha}$ denotes the mean value of $\alpha_a$. 
As before, each row of the Tables \ref{tb:ws_budget} -- \ref{tb:uncorr_budget} presents the averages for five instances.
Observe that the binary search algorithm is not as sensitive to the budget as the exact algorithms. For the exact the algorithms, while the root gap decreases with larger budget values, the solution time tends to increase, especially for the uncorrelated instances.

\begin{table}
	\footnotesize
	\centering
	\caption{Performance of the binary search for varying budgets.}
	\label{tb:ws_budget}
	\setlength{\tabcolsep}{2.5pt} 
	\begin{tabular}{c|rrr|rrr}
		\hline \hline
		&  \multicolumn{3}{c|}{ Uncorrelated } & \multicolumn{3}{c}{Correlated}  \\
		\hline
		$\eta$ & iter & time & gap & iter & time & gap \\ 
		\hline
		2 & 3.0 & 0.30 & 0.00 &  4.0 & 1.33 & 0.27 \\
		4 & 2.8 & 0.36 & 1.06  & 4.0 & 2.68 & 0.44 \\
		6 & 2.8 & 0.38 & 1.02  & 4.0 & 1.70 & 0.54 \\
		8 & 3.0 & 0.35 & 0.17  & 4.0 & 1.69 & 0.26 \\
		10 & 3.0 & 0.36 & 0.03  & 4.0 & 1.48 & 1.50 \\
		\hline
		\textbf{avg} 
		&$\BD{2.92}$ & $\BD{0.35}$ & $\BD{0.46}$ & $\BD{4.0}$ & $\BD{1.78}$ & $\BD{0.60}$ \\
		\hline \hline
	\end{tabular}
\end{table}

\begin{table}[t]
	\footnotesize
	\centering
	\caption{Performance of b\&b and b\&c for varying budgets.}
	\label{tb:uncorr_budget}
	\setlength{\tabcolsep}{2pt} 
	\begin{tabular}{c|rrrrr|rrrrrr}
		\hline \hline
				 \rule{0pt}{10pt} 	& \multicolumn{11}{c}{ Uncorrelated instances }  \\ \hline
		& \multicolumn{5}{c|}{ Cplex } & \multicolumn{6}{c}{Cplex $+$ cuts} \\ 
		\hline
		$\eta$ & rgap & stime & time & egap (\#) & nodes  & cuts & rgap  & stime & time & egap (\#) & nodes  \\ 
		\hline
		2 & 24.4 & 3 & 6 & 0.0\phantom{(0)} & 680 & 104 & 7.1 & 7 & 8 & 0.0\phantom{(0)} & 18 \\
		4 & 20.0 & 469 & 676 & 0.0\phantom{(0)} & 56,937 & 579 & 11.4 & 48 & 102 & 0.0\phantom{(0)} & 4850 \\
		6 & 18.7 & 888 & 1,412 & 0.1(1) & 123,930 & 626 & 11.5 & 94 & 124 & 0.0\phantom{(0)} & 6457 \\
		8 & 18.6 & 343 & 1,618 & 0.0\phantom{(0)} & 121,157 & 563 & 12.5 & 167 & 327 & 0.0\phantom{(0)} & 25,062 \\
		10 & 18.1 & 898 & 1,705 & 1.1(1) & 121,624 & 523 & 12.3 & 225 & 300 & 0.0\phantom{(0)} & 21,130 \\
		\hline
		{\textbf{avg}  \rule{0pt}{10pt} } &
		$\BD{20.0}$ & $\BD{520}$ & $\BD{1083}$ & $\BD{0.2(2)}$ & $\BD{84,865}$ & $\BD{479}$ & $\BD{10.9}$ & $\BD{108}$ & $\BD{172}$ & $\BD{0.0\phantom{(0)}}$ & $\BD{11,503}$ \\
		\hline 
		\rule{0pt}{10pt} 	& \multicolumn{11}{c}{ Correlated instances }  \\ \hline
		& \multicolumn{5}{c|}{ Cplex } & \multicolumn{6}{c}{Cplex $+$ cuts} \\ 
		\hline
		$\eta$ & rgap & stime & time & egap (\#) & nodes  & cuts & rgap  & stime & time & egap (\#) & nodes  \\ 
		\hline
		2 & 25.4 & 6 & 11 & 0.0\phantom{(0)} & 1,325 & 128 & 19.8 & 3 & 9 & 0.0\phantom{(0)} & 125 \\
		4 & 16.9 & 75 & 314 & 0.0\phantom{(0)} & 30,536 & 400 & 13.8 & 48 & 81 & 0.0\phantom{(0)} & 3,567 \\
		6 & 13.4 & 86 & 246 & 0.0\phantom{(0)} & 32,990 & 408 & 10.6 & 43 & 94 & 0.0\phantom{(0)} & 7,574 \\
		8 & 13.2 & 105 & 199 & 0.0\phantom{(0)} & 27,514 & 386 & 10.9 & 38 & 62 & 0.0\phantom{(0)} & 4,997 \\
		10 & 12.3 & 37 & 129 & 0.0\phantom{(0)} & 20,880 & 330 & 9.8 & 33 & 35 & 0.0\phantom{(0)} & 1,870 \\
		\hline
		{\textbf{avg}  \rule{0pt}{10pt} } &
		$\BD{16.2}$ & $\BD{61}$ & $\BD{180}$ & $\BD{0.0\phantom{(0)}}$ & $\BD{22,649}$ & $\BD{330}$ & $\BD{13.0}$ & $\BD{33}$ & $\BD{56}$ & $\BD{0.0\phantom{(0)}}$ & $\BD{3,627}$ \\
		\hline \hline
	\end{tabular}
\end{table}

\ignore{
\begin{table}[t]
	\footnotesize
	\centering
	\caption{Correlated case for varying budgets: $20 \times 20$.}
	\label{tb:corr_budget}
	\setlength{\tabcolsep}{2pt} 
	\begin{tabular}{c|rrrrr|rrrrrr}
		\hline \hline
		& \multicolumn{5}{c|}{ Cplex } & \multicolumn{6}{c}{Cplex $+$ cuts} \\ 
		\hline
		$\eta$ & rgap & stime & time & egap (\#) & nodes  & cuts & rgap  & stime & time & egap (\#) & nodes  \\ 
		\hline
		2 & 25.4 & 6 & 11 & 0.0\phantom{(0)} & 1,325 & 128 & 19.8 & 3 & 9 & 0.0\phantom{(0)} & 125 \\
		4 & 16.9 & 75 & 314 & 0.0\phantom{(0)} & 30,536 & 400 & 13.8 & 48 & 81 & 0.0\phantom{(0)} & 3,567 \\
		6 & 13.4 & 86 & 246 & 0.0\phantom{(0)} & 32,990 & 408 & 10.6 & 43 & 94 & 0.0\phantom{(0)} & 7,574 \\
		8 & 13.2 & 105 & 199 & 0.0\phantom{(0)} & 27,514 & 386 & 10.9 & 38 & 62 & 0.0\phantom{(0)} & 4,997 \\
		10 & 12.3 & 37 & 129 & 0.0\phantom{(0)} & 20,880 & 330 & 9.8 & 33 & 35 & 0.0\phantom{(0)} & 1,870 \\
		\hline
		{\textbf{avg}  \rule{0pt}{10pt} } &
		$\BD{16.2}$ & $\BD{61}$ & $\BD{180}$ & $\BD{0.0\phantom{(0)}}$ & $\BD{22,649}$ & $\BD{330}$ & $\BD{13.0}$ & $\BD{33}$ & $\BD{56}$ & $\BD{0.0\phantom{(0)}}$ & $\BD{3,627}$ \\
		\hline \hline
	\end{tabular}
\end{table}
}


Next, we present the experiments performed to test the effect of the interdiction cost parameter $\alpha$. 
New instances with $1-\epsilon = 0.95$ and grid size $20 \times 20$ are generated with
varying $\alpha_a$ drawn from integral uniform $[r, 3r]$ for $r \in \{5, 10, 15, 20, 25\}$. 
To keep the relative scales of the parameters consistent with the previous experiments,
the budget parameter $\beta$ is set to $\frac{\bar{\alpha} Y}{4}$. 
Tables \ref{tb:ws_cost} -- \ref{tb:uncorr_cost} summarize the results. The optimality gaps for the binary search algorithm
are higher for these experiments with similar run times. Both the binary search and the exact algorithms appear to be insensitive to the changes in the interdiction cost in our experiments.

\begin{table}
	\footnotesize
	\centering
	\caption{Performance of the binary search for varying interdiction costs.}
	\label{tb:ws_cost}
	\setlength{\tabcolsep}{2.5pt} 
	\begin{tabular}{c|rrr|rrr}
		\hline \hline
		&  \multicolumn{3}{c|}{ Uncorrelated } & \multicolumn{3}{c}{Correlated}  \\
		\hline
		$r$ & iter & time & gap & iter & time & gap \\ 
		\hline
		5 & 3.0 & 0.39 & 3.93 &  4.0 & 2.15 & 0.22 \\
		10 & 2.8 & 0.42 & 1.94 &  4.0 & 2.05 & 4.29 \\
		15 & 2.8 & 0.41 & 3.43 &  4.0 & 2.17 & 0.48 \\
		20 & 2.8 & 0.41 & 1.20 & 4.0 & 2.24 & 0.80 \\
		25 & 2.8 & 0.44 & 1.76 &  4.0 & 2.22 & 0.63 \\
		\hline
		\textbf{avg} 
		&$\BD{2.84}$ & $\BD{0.41}$ & $\BD{2.45}$ & $\BD{4.0}$ & $\BD{2.17}$ & $\BD{1.29}$ \\
		\hline \hline
	\end{tabular}
\end{table}

\begin{table}[t]
	\footnotesize
	\centering
	\caption{Performance of b\&b and b\&c for varying interdiction costs.}
	\label{tb:uncorr_cost}
	\setlength{\tabcolsep}{2pt} 
	\begin{tabular}{c|rrrrr|rrrrrr}
		\hline \hline
		 \rule{0pt}{10pt} 	& \multicolumn{11}{c}{ Uncorrelated instances }  \\ \hline
		& \multicolumn{5}{c|}{ Cplex } & \multicolumn{6}{c}{Cplex $+$ cuts} \\ 
		\hline
		$r$  & rgap & stime & time & egap (\#) & nodes  & cuts & rgap  & stime & time & egap (\#) & nodes  \\ 
		\hline
		5 &  22.4 & 969 & 2,405 & 0.6(2) & 299,432 & 705 & 15.1 & 296 & 428 & 0.0\phantom{(0)} & 26,404 \\
		10 &  22.5 & 1,299 & 2,552 & 1.8(2) & 282,257 & 725 & 15.4 & 802 & 944 & 0.0\phantom{(0)} & 56,725 \\
		15 &  22.4 & 1,611 & 2,383 & 3.3(3) & 267,595 & 686 & 15.3 & 815 & 1,319 & 0.5(1) & 72,508 \\
		20 &  22.2 & 1,436 & 2,905 & 4.4(4) & 279,349 & 704 & 14.9 & 336 & 775 & 0.0\phantom{(0)} & 44,972 \\
		25 &  22.3 & 1,502 & 2,905 & 4.2(4) & 339,789 & 691 & 15.2 & 576 & 985 & 0.2(1) & 57,971 \\
		\hline
		\textbf{avg} \rule{0pt}{10pt} &
		$\BD{22.4}$ & $\BD{1363}$ & $\BD{2630}$ & $\BD{2.8(15)}$ & $\BD{293,684}$ & $\BD{702}$ & $\BD{15.2}$ & $\BD{565}$ & $\BD{890}$ & $\BD{0.1(2)}$ & $\BD{51,716}$ \\ \hline
			\rule{0pt}{10pt}	& \multicolumn{11}{c}{ Correlated instances }  \\ \hline
		& \multicolumn{5}{c|}{ Cplex } & \multicolumn{6}{c}{Cplex $+$ cuts} \\ 
		\hline
		$r$  & rgap & stime & time & egap (\#) & nodes  & cuts & rgap  & stime & time & egap (\#) & nodes  \\ 
		\hline
		5 &  16.7 & 302 & 887 & 0.0\phantom{(0)} & 966,71 & 449 & 13.9 & 72 & 135 & 0.0\phantom{(0)} & 10,406 \\
		10 &  17.2 & 644 & 1,020 & 0.0\phantom{(0)} & 118,975 & 442 & 14.4 & 163 & 209 & 0.0\phantom{(0)} & 17,283 \\
		15 &  17.0 & 175 & 1,359 & 0.5(1) & 124,748 & 434 & 14.2 & 91 & 157 & 0.0\phantom{(0)} & 13,486 \\
		20 &  16.9 & 800 & 1,434 & 0.0\phantom{(0)} & 140,953 & 418 & 14.0 & 57 & 263 & 0.0\phantom{(0)} & 21,663 \\
		25 &  16.7 & 379 & 1,276 & 0.0\phantom{(0)} & 140,621 & 440 & 13.8 & 108 & 200 & 0.0\phantom{(0)} & 16,675 \\
		\hline
		\textbf{avg} \rule{0pt}{10pt} & 
		$\BD{16.9}$ & $\BD{460}$ & $\BD{1195}$ & $\BD{0.1(1)}$ & $\BD{124393}$ & $\BD{436}$ & $\BD{14.0}$ & $\BD{98}$ & $\BD{193}$ & $\BD{0.0\phantom{(0)}}$ & $\BD{15,902}$ \\
		\hline \hline
	\end{tabular}
\end{table}
\ignore{
\begin{table}[t]
	\footnotesize
	\centering
	\caption{Correlated case for varying interdiction costs.}
	\label{tb:corr_cost}
	\setlength{\tabcolsep}{2pt} 
	\begin{tabular}{c|rrrrr|rrrrrr}
		\hline \hline
		& \multicolumn{5}{c|}{ Cplex } & \multicolumn{6}{c}{Cplex $+$ cuts} \\ 
		\hline
		$r$  & rgap & stime & time & egap (\#) & nodes  & cuts & rgap  & stime & time & egap (\#) & nodes  \\ 
		\hline
		5 &  16.7 & 302 & 887 & 0.0\phantom{(0)} & 966,71 & 449 & 13.9 & 72 & 135 & 0.0\phantom{(0)} & 10,406 \\
		10 &  17.2 & 644 & 1,020 & 0.0\phantom{(0)} & 118,975 & 442 & 14.4 & 163 & 209 & 0.0\phantom{(0)} & 17,283 \\
		15 &  17.0 & 175 & 1,359 & 0.5(1) & 124,748 & 434 & 14.2 & 91 & 157 & 0.0\phantom{(0)} & 13,486 \\
		20 &  16.9 & 800 & 1,434 & 0.0\phantom{(0)} & 140,953 & 418 & 14.0 & 57 & 263 & 0.0\phantom{(0)} & 21,663 \\
		25 &  16.7 & 379 & 1,276 & 0.0\phantom{(0)} & 140,621 & 440 & 13.8 & 108 & 200 & 0.0\phantom{(0)} & 16,675 \\
		\hline
		\textbf{avg} \rule{0pt}{10pt} & 
		$\BD{16.9}$ & $\BD{460}$ & $\BD{1195}$ & $\BD{0.1(1)}$ & $\BD{124393}$ & $\BD{436}$ & $\BD{14.0}$ & $\BD{98}$ & $\BD{193}$ & $\BD{0.0\phantom{(0)}}$ & $\BD{15,902}$ \\
		\hline \hline
	\end{tabular}
\end{table}
}


Finaly, we test the performance of the binary search algorithm for larger grid sizes up to $100 \times 100$
to see how it scales up. Five instances of each size are generated as in our original set of instances. The exact algorithms are not run for these large instances; therefore, the gap is computed against the convex relaxation of the problem and, hence, it provides an upper bound on the optimality gap.
\autoref{tb:ws_large} reports the number of iterations, the time spent for the algorithm, and the percentage integrality gap, that is the gap between the upper bound found by the algorithm and the lower bound from the convex relaxation. 

Observe that the $100 \times 100$ instances have about 20,000 arcs. The correlated instances for this size could not be run due to memory limit. For the $20 \times 20$ instances, the reported upper bounds 20.73\% and 17.31\% on the optimality gap should be compared with the actual optimality gaps 1.06\% and 0.44\% in Table~\ref{tab:ws}. The large difference between the exact gap in Table~\ref{tab:ws} and igap in Table~\ref{tb:ws_large} is indicative of poor lower bounds from the convex relaxations, rather than poor upper bounds. The binary search algorithm converges in a small number of iterations for these large instances as well; however, solving quadratic 0-1 optimization problems at each iteration takes significantly longer time.

\begin{table}
	\footnotesize
	\centering
	\caption{Performance of the binary local search for larger networks.}
	\label{tb:ws_large}
	\setlength{\tabcolsep}{2.5pt} 
	\begin{tabular}{c|rrr|rrr}
		\hline \hline
		&  \multicolumn{3}{c|}{ Uncorrelated } & \multicolumn{3}{c}{Correlated}  \\
		\hline
		$p \times q$  & iter & time & igap & iter & time & igap \\ 
		\hline
		$20 \times 20$   & 2.8 & 0.36 & 20.73   & 4.0 & 2.68 & 17.31\\
		$40 \times 40$  & 3.0 & 5.33 & 26.15   & 5.0 & 80.47 & 11.81\\
		$60 \times 60$  & 3.2 & 35.66 & 27.09   & 6.0 & 502.14 & 10.12\\
		$80 \times 80$  & 3.6 & 141.96 & 27.31  & 6.0 & 3,199.11 & 8.53\\
		$100 \times 100$  & 10.2 & 4,991.3 & 31.08  & - & - & - \\
		\hline
		\textbf{avg} 
		& $\BD{4.6}$ & $\BD{1,034.92}$ & $\BD{26.47}$ & $\BD{5.3}$ & $\BD{945.97}$ & $\BD{11.94}$ \\
		\hline \hline
	\end{tabular}
\end{table}

}

\ignore{

\begin{table}
	\footnotesize
	\centering
	\caption{Performance of the binary local search for larger networks.}
	\label{tb:ws_large}
	\setlength{\tabcolsep}{1pt} 
	\begin{tabular}{c|rrrr|rrrr}
		\hline \hline
		&  \multicolumn{4}{c|}{ Uncorrelated } & \multicolumn{4}{c}{Correlated}  \\
		\hline
		$p \times q$ & time & iter & algtime & igap & time & iter & algtime & igap \\ 
		\hline
		$20 \times 20$  & 1.51 & 2.8 & 0.36 & 20.73  & 1.9 & 4.0 & 2.16 & 17.31\\
		$40 \times 40$ & 59.03 & 3.0 & 5.33 & 26.15  & 121.19 & 5.0 & 80.47 & 11.81\\
		$60 \times 60$ & 671.88 & 3.2 & 35.66 & 27.09  & 1,322.45 & 6.0 & 502.14 & 10.12\\
		$80 \times 80$ & 3,398.32 & 3.6 & 141.96 & 27.31  & 3,812.97 & 6.0 & 3,199.11 & 8.53\\
		$100 \times 100$ & 12,899.06 & 10.2 & 4,991.3 & 31.08  & 14,284.36 & - & - & - \\
		\hline
		\textbf{avg} 
		& $\BD{3,405.96}$ & $\BD{4.6}$ & $\BD{1,034.92}$ & $\BD{26.47}$ & $\BD{1,314.62}$ & $\BD{5.3}$ & $\BD{945.97}$ & $\BD{11.94}$ \\
		\hline \hline
	\end{tabular}
\end{table}
}

\exclude{ 
	\begin{table}
		\footnotesize
		\centering
		\caption{Performance of the binary local search for larger networks.}
		\label{tb:ws_large}
		\setlength{\tabcolsep}{1pt} 
		\begin{tabular}{c|rr|rr}
			\hline \hline
			&  \multicolumn{2}{c|}{ Uncorrelated } & \multicolumn{2}{c}{Correlated}  \\
			\hline
			$Y$ & iter & time & iter & time  \\ 
			\hline
			40 & 3.0 & 5.48 &  5.0 & 67.45 \\
			50 & 3.0 & 15.43 &  5.2 & 216.50 \\
			60 & 3.2 & 35.98 &  6.0 & 528.78 \\
			70 & 3.2 & 65.75 & 6.0 & 1446.10 \\
			\hline
			\textbf{avg} 
			&$\BD{3.1}$ & $\BD{30.66}$ & $\BD{5.5}$ & $\BD{564.71}$  \\
			\hline \hline
		\end{tabular}
	\end{table}

	\begin{table}
		\footnotesize
		\centering
		\caption{Performance of the binary local search for larger networks.}
		\label{tb:ws_large}
		\setlength{\tabcolsep}{1pt} 
		\begin{tabular}{c|rr|rr}
			\hline \hline
			&  \multicolumn{2}{c|}{ Uncorrelated } & \multicolumn{2}{c}{Correlated}  \\
			\hline
			$Y$ & iter & time & iter & time  \\ 
			\hline
			40 & 3.0 & 5.48 &  5.0 & 67.45 \\
			50 & 3.0 & 15.43 &  5.2 & 216.50 \\
			60 & 3.2 & 35.98 &  6.0 & 528.78 \\
			70 & 3.2 & 65.75 & 6.0 & 1446.10 \\
			80 & 3.6 & 142.40 & - & - \\ 
			\hline
			\textbf{avg} 
			&$\BD{3.2}$ & $\BD{53.01}$ & $\BD{5.5}$ & $\BD{564.71}$  \\
			\hline \hline
		\end{tabular}
\end{table}
}


\section{Conclusion} \label{sec:conclusion}

In this paper we introduce a successive quadratic optimization procedure embedded in a bisection search for finding high quality solutions to discrete mean-risk minimization problems with a conic quadratic objective. The search algorithm is applied on a non-convex upper-bounding function that provides tight values at local minima.     
Computations with the network interdiction problem with stochastic capacities indicate that 
the proposed method finds solutions within 1--\rep{4}{2}\% optimal in a small fraction of the time required by exact branch-and-bound and branch-and-cut algorithms. Although we demonstrate the approach for the network interdiction problem with stochastic capacities, since method is agnostic to the constraints of the problem, it can be applied to any \ins{0-1} optimization  problem with a mean-risk objective.

\section*{Acknowledgement} This research is supported, in part, by grant FA9550-10-1-0168 from the Office
of the Assistant Secretary of Defense for Research and Engineering.

\bibliographystyle{plainnat} 
\bibliography{interdiction}

\begin{thebibliography}{36}
\providecommand{\natexlab}[1]{#1}
\providecommand{\url}[1]{\texttt{#1}}
\expandafter\ifx\csname urlstyle\endcsname\relax
  \providecommand{\doi}[1]{doi: #1}\else
  \providecommand{\doi}{doi: \begingroup \urlstyle{rm}\Url}\fi

\bibitem[Ahmed(2006)]{ahmed:06}
S.~Ahmed.
\newblock Convexity and decomposition of mean-risk stochastic programs.
\newblock \emph{Mathematical Programming}, 106:\penalty0 433--446, 2006.

\bibitem[Ahmed and Atamt{\"u}rk(2011)]{AA:utility}
S.~Ahmed and A.~Atamt{\"u}rk.
\newblock Maximizing a class of submodular utility functions.
\newblock \emph{Mathematical Programming}, 128:\penalty0 149--169, 2011.

\bibitem[Alizadeh(1995)]{Alizadeh1995}
F.~Alizadeh.
\newblock Interior point methods in semidefinite programming with applications
  to combinatorial optimization.
\newblock \emph{SIAM Journal on Optimization}, 5:\penalty0 13--51, 1995.

\bibitem[Alizadeh and Goldfarb(2003)]{Alizadeh2003}
F.~Alizadeh and D.~Goldfarb.
\newblock Second-order cone programming.
\newblock \emph{Mathematical Programming}, 95:\penalty0 3--51, 2003.

\bibitem[Atamt{\"u}rk and G{\'o}mez(2017)]{AG:max-util}
A.~Atamt{\"u}rk and A.~G{\'o}mez.
\newblock Maximizing a class of utility functions over the vertices of a
  polytope.
\newblock \emph{Operations Research}, 65:\penalty0 433--445, 2017.

\bibitem[Atamt\"urk and G\'omez(2017)]{AG:poly}
A.~Atamt\"urk and A.~G\'omez.
\newblock Submodularity in conic quadratic mixed 0-1 optimization.
\newblock \emph{arXiv preprint arXiv:1705.05918}, 2017.
\newblock BCOL Research Report 16.02, IEOR, UC Berkeley.

\bibitem[Atamt{\"u}rk and G\'{o}mez(2017)]{AG:simplex-qp}
A.~Atamt{\"u}rk and A.~G\'{o}mez.
\newblock Simplex {QP}-based methods for minimizing a conic quadratic objective
  over polyhedra.
\newblock \emph{arXiv preprint arXiv:1706.05795}, 2017.
\newblock BCOL Research Report 17.02, IEOR, UC Berkeley.

\bibitem[Atamt\"urk and G\'omez(2018)]{AG:m-matrix}
A.~Atamt\"urk and A.~G\'omez.
\newblock Strong formulations for quadratic optimization with m-matrices and
  indicator variables.
\newblock \emph{Mathematical Programming}, 170:\penalty0 141--176, 2018.

\bibitem[Atamt\"urk and Jeon(2017)]{AJ:lifted-polymatroid}
A.~Atamt\"urk and H.~Jeon.
\newblock Lifted polymatroid inequalities for mean-risk optimization with
  indicator variables.
\newblock \emph{arXiv preprint arXiv:1705.05915}, 2017.
\newblock BCOL Research Report 17.01, IEOR, UC Berkeley.

\bibitem[Atamt{\"u}rk and Narayanan(2007)]{AN:conicmir:ipco}
A.~Atamt{\"u}rk and V.~Narayanan.
\newblock Cuts for conic mixed integer programming.
\newblock In M.~Fischetti and D.~P. Williamson, editors, \emph{Proceedings of
  the 12th International IPCO Conference}, pages 16--29, 2007.

\bibitem[Atamt{\"u}rk and Narayanan(2008)]{atamturk2008polymatroids}
A.~Atamt{\"u}rk and V.~Narayanan.
\newblock Polymatroids and mean-risk minimization in discrete optimization.
\newblock \emph{Operations Research Letters}, 36:\penalty0 618--622, 2008.

\bibitem[Atamt{\"u}rk and Narayanan(2009)]{Atamturk2009}
A.~Atamt{\"u}rk and V.~Narayanan.
\newblock The submodular 0-1 knapsack polytope.
\newblock \emph{Discrete Optimization}, 6:\penalty0 333--344, 2009.

\bibitem[Ben-Tal and Nemirovski(2000)]{BN:robust-mp}
A.~Ben-Tal and A.~Nemirovski.
\newblock Robust solutions of linear programming problems contaminated with
  uncertain data.
\newblock \emph{Mathematical Programming}, 88:\penalty0 411--424, 2000.

\bibitem[Ben-Tal and Nemirovski(2001)]{BTN:ModernOptBook}
A.~Ben-Tal and A.~Nemirovski.
\newblock \emph{Lectures on modern convex optimization: analysis, algorithms,
  and engineering applications}.
\newblock SIAM, 2001.

\bibitem[Ben-Tal et~al.(2009)Ben-Tal, El~Ghaoui, and Nemirovski]{RO-book}
A.~Ben-Tal, L.~El~Ghaoui, and A.~Nemirovski.
\newblock \emph{Robust optimization}.
\newblock Princeton University Press, 2009.

\bibitem[Bertsimas and Popescu(2005)]{bertsimas.popescu:05}
D.~Bertsimas and I.~Popescu.
\newblock Optimal inequalities in probability theory: A convex optimization
  approach.
\newblock \emph{SIAM Journal on Optimization}, 15:\penalty0 780--804, 2005.

\bibitem[Bienstock and Muratore(2000)]{BM:surv-cuts}
D.~Bienstock and G.~Muratore.
\newblock Strong inequalities for capacitated survivable network design
  problems.
\newblock \emph{Mathematical Programming}, 89:\penalty0 127--147, 2000.

\bibitem[Birge and Louveaux(2011)]{Birge:SPbook}
J.~R. Birge and F.~Louveaux.
\newblock \emph{Introduction to Stochastic Programming}.
\newblock Springer, 2011.

\bibitem[{\c{C}}ezik and Iyengar(2005)]{CI:cmip}
M.~T. {\c{C}}ezik and G.~Iyengar.
\newblock Cuts for mixed 0-1 conic programming.
\newblock \emph{Mathematical Programming}, 104:\penalty0 179--202, 2005.

\bibitem[Cormican et~al.(1998)Cormican, Morton, and
  Wood]{cormican1998stochastic}
K.~J. Cormican, D.~P. Morton, and R.~K. Wood.
\newblock Stochastic network interdiction.
\newblock \emph{Operations Research}, 46:\penalty0 184--197, 1998.

\bibitem[Ghaoui et~al.(2003)Ghaoui, Oks, and Oustry]{ghaoui.etal:03}
L.~E. Ghaoui, M.~Oks, and F.~Oustry.
\newblock Worst-case value-at-risk and robust portfolio optimization: A conic
  programming approach.
\newblock \emph{Operations Research}, 51:\penalty0 543--556, 2003.

\bibitem[Hassin and Tamir(1989)]{hassin1989maximizing}
R.~Hassin and A.~Tamir.
\newblock Maximizing classes of two-parameter objectives over matroids.
\newblock \emph{Mathematics of Operations Research}, 14:\penalty0 362--375,
  1989.

\bibitem[Held et~al.(2005)Held, Hemmecke, and Woodruff]{HHW:stoc-interdiction}
H.~Held, R.~Hemmecke, and D.~L. Woodruff.
\newblock A decomposition algorithm applied to planning the interdiction of
  stochastic networks.
\newblock \emph{Naval Research Logistics}, 52:\penalty0 321--328, 2005.

\bibitem[Hiriart-Urruty and Lemar{\'e}chal(2013)]{hiriart2013convex}
J.-B. Hiriart-Urruty and C.~Lemar{\'e}chal.
\newblock \emph{Convex analysis and minimization algorithms I: Fundamentals}.
\newblock Springer, 2013.

\bibitem[Ishii et~al.(1981)Ishii, Shiode, Nishida, and Namasuya]{Ishii1981}
H.~Ishii, S.~Shiode, T.~Nishida, and Y.~Namasuya.
\newblock Stochastic spanning tree problem.
\newblock \emph{Discrete Applied Mathematics}, 3:\penalty0 263--273, 1981.

\bibitem[Janjarassuk and Linderoth(2008)]{janjarassuk2008reformulation}
U.~Janjarassuk and J.~Linderoth.
\newblock Reformulation and sampling to solve a stochastic network interdiction
  problem.
\newblock \emph{Networks}, 52:\penalty0 120--132, 2008.

\bibitem[Lei et~al.(2018)Lei, Shen, and Song]{LSS:interdiction}
Xiao Lei, Siqian Shen, and Yongjia Song.
\newblock Stochastic maximum flow interdiction problems under heterogeneous
  risk preferences.
\newblock \emph{Computers \& Operations Research}, 90:\penalty0 97--109, 2018.

\bibitem[Lobo et~al.(1998)Lobo, Vandenberghe, Boyd, and Lebret]{Lobo1998}
M.~S. Lobo, L.~Vandenberghe, S.~Boyd, and H.~Lebret.
\newblock Applications of second-order cone programming.
\newblock \emph{Linear Algebra \& its Applications}, 284:\penalty0 193--228,
  1998.

\bibitem[Nesterov and Todd(1998)]{Nesterov1998}
Y.~E Nesterov and M.~J. Todd.
\newblock Primal-dual interior-point methods for self-scaled cones.
\newblock \emph{SIAM Journal on Optimization}, 8:\penalty0 324--364, 1998.

\bibitem[Nikolova(2009)]{nikolova2009strategic}
E.~Nikolova.
\newblock \emph{Strategic Algorithms}.
\newblock PhD thesis, Massachusetts Institute of Technology, 2009.

\bibitem[Rajan and Atamt{\"u}rk(2002)]{RA:review}
D.~Rajan and A.~Atamt{\"u}rk.
\newblock Survivable network design : {R}outing of flows and slacks.
\newblock In G.~Anandalingam and S.~Raghavan, editors, \emph{Telecommunications
  Network Design and Management}, pages 65--81. Kluwer Academic Publishers,
  2002.

\bibitem[Royset and Wood(2007)]{royset2007solving}
J.~O. Royset and R.~K. Wood.
\newblock Solving the bi-objective maximum-flow network-interdiction problem.
\newblock \emph{INFORMS Journal on Computing}, 19:\penalty0 175--184, 2007.

\bibitem[Schrijver(2003)]{book:S:co}
Alexander Schrijver.
\newblock \emph{Combinatorial optimization: polyhedra and efficiency},
  volume~24.
\newblock Springer Science \& Business Media, 2003.

\bibitem[Shen et~al.(2003)Shen, Coullard, and Daskin]{shen2003joint}
Z.-J.~M. Shen, C.~Coullard, and M.~S. Daskin.
\newblock A joint location-inventory model.
\newblock \emph{Transportation Science}, 37:\penalty0 40--55, 2003.

\bibitem[Smith et~al.(2013)Smith, Prince, and Geunes]{Smith2013}
J.~C. Smith, M.~Prince, and J.~Geunes.
\newblock Modern network interdiction problems and algorithms.
\newblock In P.~M. Pardalos, D.-Z. Du, and R.~L. Graham, editors,
  \emph{Handbook of Combinatorial Optimization}, pages 1949--1987. Springer,
  2013.

\bibitem[Wood(1993)]{wood1993deterministic}
R.~K. Wood.
\newblock Deterministic network interdiction.
\newblock \emph{Mathematical and Computer Modelling}, 17:\penalty0 1--18, 1993.

\end{thebibliography}

\end{document}